\documentclass[11pt]{article}
\usepackage{amsmath}
\usepackage{dsfont}
\usepackage{mathrsfs}
\usepackage{amsmath,amssymb}
\usepackage{amsfonts}
\usepackage{hyperref}
\usepackage{amsthm}
\usepackage{graphicx}
\usepackage{subfigure}
\usepackage{xcolor}

\hfuzz=\maxdimen
\theoremstyle{definition}
\newtheorem{lemma}{Lemma}[section]

\newtheorem{theorem}[lemma]{Theorem}
\newtheorem{corollary}[lemma]{Corollary}

\newtheorem{remark}{Remark}

\numberwithin{equation}{section}

\DeclareFixedFont{\Acknowledgment}{OT1}{cmr}{bx}{n}{14pt}
\begin{document}

\title{\bf Positive Solutions of $p$-th Yamabe Type Equations on Graphs}
\author{Xiaoxiao Zhang, Aijin Lin}
\date{}
\maketitle

\begin{abstract}
Let $G=(V,E)$ be a finite connected weighted graph, and assume $1\leq\alpha\leq p\leq q$. In this paper, we consider the following $p$-th Yamabe type equation
$$-\Delta_pu+hu^{q-1}=\lambda fu^{\alpha-1}.$$
on $G$, where $\Delta_p$ is the $p$-th discrete graph Laplacian, $h\leq0$ and $f>0$ are real functions defined on all vertices of $G$. Instead of the approach in \cite{Ge3}, we adopt a new approach, and prove that the above equation always has a positive solution $u>0$ for some constant $\lambda\in\mathds{R}$. In particular, when $q=p$ our result generalizes the main theorem in \cite{Ge3} from the case of $\alpha\geq p>1$ to the case of $1\leq\alpha\leq p$. It's interesting that our new approach can also work in the case of $\alpha\geq p>1$.
\end{abstract}
\vspace{12pt}
\let\thefootnote\relax \footnotetext {The first author
is supported by the National Natural Science Foundation of China
(Grant No. 11431003). The second author is supported by the National Natural Science Foundation of China
(Grant No. 11401578).}


\section{Introduction}

As is known that, let $(M^m,g)$ be a closed Riemannian manifold of dimension $m(\geq3)$, the Yamabe problem consists of finding metrics of constant scalar curvature in $[g]=\{fg|f:M\rightarrow \mathbb{R}_{>0}\}$, the conformal class of $g$. A metric $\bar{g}=f^{p_m-2}g$ conformal to $g$ has constant scalar curvature $s\in \mathbb{R}$ if and only if the positive function $f$ satisfies the \emph{Yamabe\ equation} corresponding to $g$:
$$-a_m\Delta_gf+s_gf=sf^{p_m-2}, $$
where $\Delta_g$ denotes the Laplace-Beltrami operator corresponding to $g$, and $s_g$ denotes the scalar curvature of $g$, $a_m=\frac{4(m-1)}{m-2}$, $p_m=\frac{2m}{m-2}$ is the Sobolev critical exponent.

In the case of closed manifolds it was proved that at least one solution exists following a program introduced by H. Yamabe in \cite{Yamabe}. Solutions of the Yamabe equations are critical points of the Hilbert-Einstein function $S$ on the space of Riemannian metrics on $M$ restricted to $[g]$, the conformal class of $g$,
$$S(g)=\frac{\int_Ms_gd\mu_g}{Vol(M,g)^{2/p_m}}.$$
This problem was also studied by Trudinger \cite{Tru}, Aubin \cite{Aubin}, and completely solved by Schoen \cite{Schoen}.

Recently, there are tremendous work concerning the discrete weighted Laplacians and various equations on graphs, among those we refer readers to \cite{BHJ,CL2,CLC,FHY,GLY,GLY2,GLY3,Ge1,Ge2,Ge3,Ge4,HKL,H,LW,WZ}. Particularly, there have been some works on dealing with Yamabe type equations on graphs, on which please refer to \cite{GLY2,GLY3,Ge3,Ge4}. Grigor'yan, Lin and Yang \cite{GLY3} first studied a Yamabe type equation on a finite graph as follows
 \begin{equation}\label{YE1}
 -\Delta u+hu=|u|^{\alpha-2}u,\quad \alpha>2,
 \end{equation}
  where $\Delta$ is a usual discrete graph Laplacian, and $h$ is a positive function defined on the vertices. They show that the equation (\ref{YE1}) always has a positive solution. Inspired by their work, Ge \cite{Ge3} studied a Yamabe type equation on a finite graph, that is
  \begin{equation}\label{G}
    -\Delta_pu+hu^{p-1}=\lambda fu^{\alpha-1}.
  \end{equation}
We state the main result in \cite{Ge3} as follows
\begin{theorem}\label{G1}
Let $G=(V,E)$ be a finite connected graph. Given $h, f\in C(V)$ with $h\leq0, f>0$. Assume $\alpha\geq p>1$. Then the following $p$-th Yamabe equation
\begin{equation*}
-\Delta_pu+hu^{p-1}=\lambda fu^{\alpha-1}
\end{equation*}\label{def-Yamabe-equ-p}
on $G$ always has a positive solution $u$ for some constant $\lambda\in\mathds{R}$.
\end{theorem}

From this result, one naturally wants to know:\\

\textbf{Problem I}: Can one solve the $p$-th Yamabe equation (\ref{G}) for $1\leq\alpha\leq p$ ?\\

To answer this question is our main purpose of this paper.

\begin{remark}
Ge and Jiang \cite{Ge4} also studied a following Yamabe type equation
  \begin{equation*}\label{alpha>p}
    -\Delta_pu+hu^{p-1}=gu^{\alpha-1},\ u>0,
  \end{equation*}
  on an infinite graph.
The main result in \cite{Ge4} is as follows
 \begin{theorem}
  Consider the $p$-th Yamabe equation (\ref{G}) on a connected, infinite and locally finite graph $G$ with $\alpha>p\geq2$. Assume $g\geq0$ and $g$ is bounded from above, $h$ satisfies $\inf_{x\in V}h(x)>0$ and $\inf_{x\in V}h(x)\mu(x)>0$. Further assume $h^{-1}\in L^{\delta}(V)$ for some $\delta>0$ (or $h(x)\rightarrow\infty$ when $x\rightarrow\infty$), then (\ref{G}) has a positive solution.
  \end{theorem}
 From this result one still needs to know:
Can one solve this $p$-th Yamabe equation (\ref{G}) under the assumption $2<\alpha\leq p$ ?\par
We will solve this problem in our another paper \cite{ZL}.
\end{remark}

Next note that Grigor'yan, Lin and Yang's pioneer paper \cite{GLY3} also studied a similar Yamabe type equation as follows
\begin{equation}\label{L}
-\Delta_pu+h|u|^{p-2}u=f(x,u), \;\;p>1,
\end{equation}
on a finite graph under the assumption $h>0$.
They show that the equation (\ref{L}) always has a positive solution under certain assumptions about $f(x,u)$. It is remarkable that their $\Delta_p$ considered in the equation (\ref{L}) is different from ours when $p\neq2$.\par
If we replace the $p$-th Laplacian $\Delta_p$ in equation (\ref{G}) with $\Delta_p$ in equation (\ref{L}), then we can ask\\

\textbf{Problem II}: Do positive solutions of the $p$-th Yamabe equation (\ref{G}) still exist ?\\

In this paper we attempt to solve both problems above on a finite graph. The answers turn out be interesting.
 In fact, we can prove that the equation (\ref{G}) actually has a positive solution $u>0$ in the case of $1\leq\alpha\leq p$. Note that Ge's approach in \cite{Ge3} can't succeed in solving Problem I any more, so we have to adopt a new approach. Let's outline our approach to solution of Problem I. \par
 First, we study the following Yamabe type equation
  \begin{equation*}
    -\Delta_pu+\mu hu^{q-1}=\lambda fu^{\alpha-1},
  \end{equation*}
 in the case of $1\leq\alpha\leq p\leq q$.\par
  Note that in this equation we add a new constant $\mu$ and extend exponents of the second term from $p$ to any integer $q$ satisfying $q\geq p$.
 In order to study the equation above, we have to make a transformation as
  \begin{equation*}
    -\Delta_pu-\lambda fu^{\alpha-1}=-\mu hu^{q-1}.
  \end{equation*}
 Otherwise our derivation can't go on.
 Then by using the constrained minimization technique, we can prove that, for any constant $\lambda\in \mathbb{R}$, the equation above
 actually has a positive solution $\hat{u}>0$ at least for some constant $\mu=\mu(\lambda)\in\mathbb{R}$.\par
 Then, set $\tilde{\lambda}<0$, we can construct a positive function $u=u(\hat{u},\tilde{\lambda})$ depending on $\hat{u},\tilde{\lambda}$ by using the scaling technique. In fact, we can prove that this function $u=u(\hat{u},\tilde{\lambda})$ is exactly a positive solution of the Yamabe type equation
   \begin{equation*}
    -\Delta_pu+ hu^{q-1}=\lambda fu^{\alpha-1}.
  \end{equation*}
 for some constant $\lambda=\lambda(\hat{u},\tilde{\lambda})$.\par
Note that here the constant $\mu$ added in previous equation disappear at all, which highly depends on the condition $1\leq\alpha\leq p\leq q$.\par
Finally take $q=p$, we completely solve Problem I.\par
The key point of our approach is the assumption: $1\leq\alpha\leq p\leq q$. In fact, our approach is still successful in the case of $1\leq q\leq p\leq \alpha$, where taking $q=p$ particularly, we can get the main theorem in \cite{Ge3}. In this respect, we find that our approach is more generic.\par

   As for Problem II, we observe that
the difference between the two $p$-th Laplacian $\Delta_p$ in equation (\ref{G}) and in equation (\ref{L}) is not essential, therefore we can similarly prove that equation (\ref{G}) actually has a positive solution. \par
We organize this paper as follows: In section 2, we introduce some notions on graphs and state our main results. In section 3, we give Sobolev imbedding. Section 4 is devoted to prove an important theorem which is key to solve both problems above. In section 5, we completely solve Problem I and Problem II.\\

\section{Settings and main results}
Let $G=(V,E)$ be a finite graph, where $V$ denotes the vertex set and $E$ denotes the edge set. Fix a vertex measure $\mu:V\rightarrow(0,+\infty)$ and an edge measure $\omega:E\rightarrow(0,+\infty)$ on $G$. The edge measure $\omega$ is assumed to be symmetric, that is, $\omega_{ij}=\omega_{ji}$ for each edge $i\thicksim j$.

Denote $C(V)$ as the set of all real functions defined on $V$, then $C(V)$ is a finite dimensional linear space with the usual function additions and scalar multiplications. For any $p>1$, the $p$-th discrete graph Laplacian $\Delta_p:C(V)\rightarrow C(V)$ is
\begin{equation*}
\Delta_pf_i=\frac{1}{\mu_i}\sum\limits_{j\thicksim i}\omega_{ij}|f_j-f_i|^{p-2}(f_j-f_i)
\end{equation*}
for any $f\in C(V)$ and $i\in V$. $\Delta_p$ is a nonlinear operator when $p\neq2$ (see \cite{Ge3} for more properties about $\Delta_p$).
Now we can state an important theorem as follows
\begin{theorem}\label{th1}
Let $G=(V,E)$ be a connected finite graph. Given $h,f\in C(V)$ with $h\leq0$, $f>0.$ Assume $1\leq\alpha\leq p\leq q$. Then for any constant $\lambda\in \mathbb{R}$, the following $p$-th Yamabe equation
\begin{equation}\label{L1}
  -\Delta_pu+\mu hu^{q-1}=\lambda fu^{\alpha-1}
\end{equation}
on $G$ always has a positive solution $u>0$ for some constant $\mu=\mu(\lambda)\in \mathbb{R}$.
\end{theorem}
If we want to get rid off $\mu$, then we have
\label{th2}
\begin{theorem}\label{th2}
Let $G=(V,E)$ be a connected finite graph. Given $h,f\in C(V)$ with $h\leq0$, $f>0.$ Assume $1\leq\alpha\leq p\leq q$. Then the following $p$-th Yamabe equation
\begin{equation}\label{L2}
  -\Delta_pu+hu^{q-1}=\lambda fu^{\alpha-1}
\end{equation}
on $G$ always has a positive solution $u>0$ for some constant $\lambda\in \mathbb{R}$.
\end{theorem}
This theorem actually generalizes the Problem I and gives a positive answer. As a corollary, let's take $q=p$, we get the following
\begin{corollary}\label{c1}
Let $G=(V,E)$ be a connected finite graph. Given $h,f\in C(V)$ with $h\leq0$, $f>0.$ Assume $1\leq\alpha\leq p$. Then the following $p$-th Yamabe equation
\begin{equation*}\label{Y-finite}
  -\Delta_pu+hu^{p-1}=\lambda fu^{\alpha-1}
\end{equation*}
on $G$ always has a positive solution $u>0$ for some constant $\lambda\in \mathbb{R}$.
\end{corollary}
This corollary completely solves Problem I. Further, combining this corollary with Theorem \ref{G1}, we have
\begin{corollary}\label{c2}
  Let $G=(V,E)$ be a connected finite graph. Given $h,f\in C(V)$ with $h\leq0$, $f>0.$ Assume $p>1$, $\alpha\geq1$. Then the following $p$-th Yamabe equation
\begin{equation*}
  -\Delta_pu+hu^{p-1}=\lambda fu^{\alpha-1}
\end{equation*}
on $G$ always has a positive solution $u$ for some constant $\lambda\in \mathbb{R}$.
\end{corollary}
Next we consider the $p$-th Laplacian $\Delta_p$ (please refer to \cite{GLY3}) in equation (\ref{L}) which is defined in distributional sense by
\begin{equation}\label{LY1}
  \int_V(\Delta_pu)\phi d\mu= -\int_V|\nabla u|^{p-2}\Gamma(u,\phi)d\mu, \forall \phi \in C_{c}(V),
\end{equation}
where $C_{c}(V)$ denotes the set of all functions with compact support, obviously $C_{c}(V)=C(V)$ when G is a finite graph. $\Gamma(u,\phi), |\nabla u|$ are defined respectively by
\begin{eqnarray}\label{Gama}
 \Gamma(u,\phi)_{i}=\frac{1}{2\mu_{i}}\sum\limits_{j\thicksim i}\omega_{ij}(u_{j}-u_{i})(\phi_j-\phi_i), \forall i \in V,
\end{eqnarray}
and
\begin{equation}\label{tidu}
|\nabla u|_{i}=\sqrt{\Gamma(u,u)_{i}}, \forall i \in V.
\end{equation}
Point-wisely, $\Delta_p$ can be written as
\begin{eqnarray}\label{PL}
\Delta_pu_{i}=\frac{1}{2\mu_{i}}\sum\limits_{j\thicksim i}(|\nabla u|^{p-2}_{j}+|\nabla u|^{p-2}_{i})\omega_{ij}(u_{j}-u_{i}), \forall i \in V.
\end{eqnarray}
Applying this $p$-th Laplacian $\Delta_p$ we can study the $p$-th Yamabe equation
\begin{equation*}\label{Y-finite}
  -\Delta_pu+hu^{p-1}=\lambda fu^{\alpha-1}
\end{equation*}
again and we obtain
\begin{theorem}\label{th3}
Let $G=(V,E)$ be a connected finite graph. Given $h,f\in C(V)$ with $h\leq0$, $f>0.$ Assume $p>1, \alpha\geq1$. Then the following $p$-th Yamabe equation
\begin{equation*}\label{Y-finite}
  -\Delta_pu+hu^{p-1}=\lambda fu^{\alpha-1}
\end{equation*}
on $G$ always has a positive solution $u>0$ for some constant $\lambda\in \mathbb{R}$.
\end{theorem}

\section{Sobolev embedding}
\label{sect-preliminary-lemma}
For any $f\in C(V)$, define the integral of $f$ over $V$ with respect to the vertex weight $\mu$ by
$$\int_Vfd\mu=\sum\limits_{i\in V}\mu_if_i.$$
Set $\mathrm{Vol}(G)=\int_Vd\mu$. \par
For any $f\in C(V)$, define
\begin{equation}\label{p-int}
\int_V|\nabla f|^pd\mu=\sum\limits_{i\thicksim j}\omega_{ij}|f_j-f_i|^p,
\end{equation}
where $|\nabla f|$ is defined as
\begin{equation}\label{ptd}
 |\nabla f_i|=\Big(\frac{1}{2\mu_i}\sum_{j\thicksim i}\omega_{ij}\big|f_j-f_i\big|^p\Big)^{1/p}, \forall i\in V.
\end{equation}
Next we consider the Sobolev space $W^{1,\,p}$ on the graph $G$. Define
$$W^{1,\,p}(G)=\left\{u\in C(V):\int_V|\nabla u|^pd\mu+\int_V|u|^pd\mu<+\infty\right\},$$
and
$$\|u\|_{W^{1,\,p}(G)}=\left(\int_V|\nabla u|^pd\mu+\int_V|u|^pd\mu\right)^{\frac{1}{p}}.$$
Since $G$ is a finite graph, then
$W^{1,\,p}(G)$ is exactly $C(V)$, a finite dimensional linear space. This implies the following Sobolev embedding \cite{Ge3}
\begin{lemma}\label{S}(Sobolev embedding)
Let $G=(V,E)$ be a finite graph. The Sobolev space $W^{1,\,p}(G)$ is pre-compact. Namely, if $\{\varphi_n\}$ is bounded in $W^{1,\,p}(G)$, then there exists some $\varphi\in W^{1,\,p}(G)$ such that up to a subsequence, $\varphi_n\rightarrow\varphi$ in $W^{1,\,p}(G)$.
\end{lemma}
\begin{remark}\label{convergence}
The convergence in $W^{1,\,p}(G)$ is in fact pointwise convergence.
\end{remark}

\section{Proof of Theorem \ref{th1}}
In this section we focus on proving Theorem \ref{th1}. Frist we consider a constrained minimization problem for the following energy functional
\begin{equation}\label{EF}
  E(u)=\frac{1}{p} \int_V|\nabla u|^pd\mu-\frac{\lambda}{\alpha}\int_Vfu^{\alpha}d\mu,
\end{equation}
on the space $C(V)$, restricted to the subset\\
$$M=\{u\in C(V): \frac{1}{q}\int_V hu^{q}d\mu=-1\}$$
Define
$$\beta=\inf_{u\in M}\{E(u):u\geq0, u\not\equiv0\}.$$
We will find a positive solution of the equation (\ref{L1}) step by step as follows.\\

\textbf{Step 1}. $E(u)$ is bounded below. \par
Let's first prove the following two lemmas
\begin{lemma}\label{lemma1}
Let $G=(V,E)$ be a connected finite graph. Given $h,f\in C(V)$ with $h\leq0$. Assume $1\leq\alpha\leq p\leq q$, $\forall u\in M, u\geq0$, then we have the following inequalities
\begin{equation}\label{inequ1}
 \|u\|^{q}_{q}\leq\frac{q}{(-h)_{m}}, \|u\|^{\alpha}_{q}\leq[\frac{q}{(-h)_{m}}]^{\frac{\alpha}{q}}
\end{equation}
where $(-h)_m=\min\limits_{i\in V}(-h_i)$.
\end{lemma}
\begin{proof}
$\forall u\in M, u\geq0$ and $h\leq0$, we have
 \begin{equation*}
q=-\int_V hu^{q}d\mu\geq(-h)_m\int_V u^{q}d\mu=(-h)_m\|u\|^{q}_{q}.
\end{equation*}
Hence
\begin{equation*}
\|u\|^{q}_{q}\leq\frac{q}{(-h)_{m}}, \|u\|^{\alpha}_{q}\leq[\frac{q}{(-h)_{m}}]^{\frac{\alpha}{q}}.
\end{equation*}
$\hfill\Box$
\end{proof}

\begin{lemma}\label{lemma2}
Let $G=(V,E)$ be a connected finite graph. Given $f\in C(V)$ with $f>0.$ Assume $1\leq\alpha\leq p\leq q$, $\forall u\in C(V), u\geq0$. Then we have the following inequality
\begin{equation}\label{inequ2}
0\leq\int_V fu^{\alpha}d\mu\leq f_{M}\mathrm{Vol}(G)^{1-\frac{\alpha}{q}}\|u\|_{q}^{\alpha}
\end{equation}
where $f_M=\max\limits_{i\in V}(f_i)$, and $\mathrm{Vol}(G)=\int_V d\mu$.
\end{lemma}
\begin{proof}
$\forall u\in C(V), u\geq0$ and $f>0$, by $\textrm{H}$$\ddot{\textrm{o}}$$\textrm{lder}$ inequality we have
 \begin{equation*}
 \int_V fu^{\alpha}d\mu\leq f_M\|1\cdot u^{\alpha}\|_{1}\leq f_{M} \|1\|_{\frac{q}{q-\alpha}}\|u^{\alpha}\|_{\frac{q}{\alpha}}\\
 =f_{M}\mathrm{Vol}(G)^{1-\frac{\alpha}{q}}\|u\|_{q}^{\alpha}.
\end{equation*}
$\hfill\Box$
\end{proof}
Now we can prove that the energy functional $E(u)$ is bounded below for all $u\geq0$, $u\not\equiv0$. Hence $\beta=\inf_{u\in M}\{E(u):u\geq0, u\not\equiv0\}$ and $\beta\in\mathbb{R}$.

\begin{theorem}\label{theorem1}
Let $G=(V,E)$ be a connected finite graph. Given $h,f\in C(V)$ with $h\leq0$, $f>0.$ Assume $1\leq\alpha<p$. Then for $\forall u\in M, u\geq0$, then the energy functional $E(u)$ is bounded below by a constant, namely we have the following inequality
\begin{equation}\label{E-inf}
 E(u)\geq C_{\lambda,\alpha,q,h,f,G},
\end{equation}
where $C_{\alpha,q,h,f,G}:=-\frac{|\lambda|}{\alpha}f_{M}[\frac{q}{(-h)_{m}}]^{\frac{\alpha}{q}}\mathrm{Vol}(G)^{1-\frac{\alpha}{q}}\leq 0$ is a constant depending only on the information of $\lambda,\alpha,q,h,f,G$.
\end{theorem}
\begin{proof}
Combining the two lemmas above, we have
 \begin{equation*}
 \int_V fu^{\alpha}d\mu\leq f_{M}\mathrm{Vol}(G)^{1-\frac{\alpha}{q}}\|u\|_{q}^{\alpha}\leq f_{M}\mathrm{Vol}(G)^{1-\frac{\alpha}{q}}
 [\frac{q}{(-h)_{m}}]^{\frac{\alpha}{q}}.
\end{equation*}
Hence
 \begin{equation*}
 E(u)\geq -\frac{\lambda}{\alpha}\int_Vfu^{\alpha}d\mu\geq -\frac{|\lambda|}{\alpha}f_{M}[\frac{q}{(-h)_{m}}]^{\frac{\alpha}{q}}\mathrm{Vol}(G)^{1-\frac{\alpha}{q}}.
\end{equation*}
$\hfill\Box$
\end{proof}

\textbf{Step 2}. There exists a $\hat{u}\geq0, \hat{u}\not\equiv0$ such that $\beta=E(\hat{u})$.\par
 To find such $\hat{u}$, we choose $u_n\in M, u_n\geq0$, satisfying $$E(u_n)\rightarrow\beta$$
as $n\rightarrow\infty$ and
$$\int_Vhu_n^{q}d\mu=-q.$$\\
Further we can assume  $E(u_n)\leq 1+\beta$ for all $n$, then we have

\begin{theorem}\label{theorem1}
Let $G=(V,E)$ be a connected finite graph. Given $h,f\in C(V)$ with $h\leq0$, $f>0.$ Assume $1\leq\alpha\leq p\leq q$. Let $u_{n}$ satisfy the above conditions, then
$\{u_{n}\}$ is bounded in $W^{1,\,p}(G)$. In fact, we have the following estimate
\begin{equation*}\label{Y-finite}
 \|u_{n}\|_{W^{1,\,p}(G)}\leq C_{\lambda,\alpha,\beta,p,q,h,f,G}
\end{equation*}
where $C_{\lambda,\alpha,\beta,p,q,h,f,G}:=p(1+\beta)+\frac{p|\lambda| f_{M}}{\alpha}[\frac{q}{(-h)_{m}}]^{\frac{\alpha}{q}}\mathrm{Vol}(G)^{1-\frac{\alpha}{q}}
+[\frac{q}{(-h)_{m}}]^{\frac{p}{q}}\mathrm{Vol}(G)^{1-\frac{p}{q}}$ is a constant depending only on the information of $\lambda,\alpha,\beta,p,q,h,f,G$.
\end{theorem}
\begin{proof}
Observe that the Sobolev norm $\|\cdot\|_{W^{1,\,p}(G)}$ is related to the energy functional $E(u)$, then by (\ref{inequ1}) and (\ref{inequ2}) we get
\begin{eqnarray*}
\begin{aligned}
\|u_{n}\|_{W^{1,\,p}(G)}=&\int_V|\nabla u_{n}|^pd\mu+\int_V|u_{n}|^{p}d\mu\\
=&\;p(\frac{1}{p} \int_V|\nabla u|^pd\mu-\frac{\lambda}{\alpha}\int_Vfu^{\alpha}d\mu)+\frac{p\lambda}{\alpha}\int_Vfu^{\alpha}d\mu+ \int_V|u_{n}|^{p}d\mu\\
\leq&\;pE(u_{n})+\frac{p\lambda}{\alpha}\int_Vfu^{\alpha}d\mu+\|u_{n}\|_{q}^{p}\mathrm{Vol}(G)^{1-\frac{p}{q}}\\
\leq &\;p(1+\beta)+\frac{p|\lambda| f_{M}}{\alpha}[\frac{q}{(-h)_{m}}]^{\frac{\alpha}{q}}\mathrm{Vol}(G)^{1-\frac{\alpha}{q}}
+[\frac{q}{(-h)_{m}}]^{\frac{p}{q}}\mathrm{Vol}(G)^{1-\frac{p}{q}}\\
=&\;C_{\lambda,\alpha,\beta,p,q,h,f,G}.
\end{aligned}
\end{eqnarray*}$\hfill\Box$
\end{proof}
 Therefore $\{u_{n}\}$ is bounded in $W^{1,\,p}(G)$. Then by Lemma \ref{S}, there exists some $\hat{u}\in W^{1,\,p}(G)=C(V)$ such that up to a subsequence, $u_n\rightarrow \hat{u}$
in $W^{1,\,p}(G)$. We may also denote this subsequence as $u_n$. Note $u_n\geq0$ and $\int_Vhu_n^{q}d\mu=-q$, let $n\rightarrow+\infty$, we know
$\hat{u}\geq0$ and
\begin{equation*}\label{ef}
\int_Vh\hat{u}^{q}d\mu=-q\neq 0,
\end{equation*}
hence $\hat{u} \in M$. This also implies that $\hat{u}\not\equiv0$. Since $G$ is finite graph, by Remark \ref{convergence} the convergence in the Sobolev space $W^{1,\,p}(G)$ is in fact pointwise convergence. Moreover,
 because $G$ is a finite graph, so the energy functional $E(u)$ is actually continuous, which yields that $E(u)$ attains its infimum exactly at the point $\hat{u}\in M$. Therefore we have $$\beta=E(\hat{u}).$$\\

\textbf{Step 3}. $\hat{u}\geq0, \hat{u}\not\equiv0$ is exactly a nontrivial solution of the Yamabe type equation (\ref{L1}) \\

Recall the constrained minimization problem for the following energy functional
\begin{equation*}\label{ef}
  E(u)=\frac{1}{p} \int_V|\nabla u|^pd\mu-\frac{\lambda}{\alpha}\int_Vfu^{\alpha}d\mu,
\end{equation*}
on the space $C(V)$, restricted to the subset\\
$$M=\{u\in C(V): \frac{1}{q}\int_V hu^{q}d\mu=-1\}.$$
Now we derive the Euler-Lagrange equations by the Lagrange multiplier rule.
First set
\begin{equation*}
G(u)=\frac{1}{q}\int_Vhu^{q}d\mu+1.
\end{equation*}
It is easy to see that both $E(u)$ and $G(u)$ are continuously Fr\'{e}chet differentiable, and the Fr\'{e}chet derivatives $DE(u), DG(u)$ are respectively given by
\begin{eqnarray*}
<v, DE(u)>&=&\int_{V}(|\nabla u|^{p-1}\nabla v-\lambda fu^{\alpha-1}v)d\mu\\
&=&\int_{V}(-\Delta_{p}u-\lambda fu^{\alpha-1})vd\mu, \forall v\in C(V),
\end{eqnarray*}
and $$<v,DG(u)>=\int_Vhu^{q-1}vd\mu, \forall v\in C(V).$$
Then let $$\Phi(t,\mu)=E(u+tv)+\mu G(u+tv), \forall v\in C(V),$$  by direct calculation we have
\begin{equation*}
\frac{d\Phi(t,\mu)}{dt}|_{t=0}=<v,DE(u)+\mu DG(u)>\\
=\int_V(-\Delta_{p}u-\lambda fu^{\alpha-1}+\mu hu^{q-1})vd\mu.
\end{equation*}
Therefore if $u$ is a critical point with regard to the above constrained minimization problem, then $u$ satisfies the following Euler-Lagrange equations
\begin{equation*}
\int_V(-\Delta_{p}u-\lambda fu^{\alpha-1}+\mu hu^{q-1})vd\mu=0, \forall v\in C(V).
\end{equation*}
In particular, the infimum point $\hat{u}$ satisfies the Euler-Lagrange equations
\begin{equation}\label{fin}
\int_V(-\Delta_{p}\hat{u}-\lambda f\hat{u}^{\alpha-1}+\mu h\hat{u}^{q-1})vd\mu=0, \forall v\in C(V).
\end{equation}

This implies $\hat{u}\geq0$ is a nontrivial solution of the Yamabe type equation (\ref{L1})
\begin{equation*}
-\Delta_{p}\hat{u}-\lambda f\hat{u}^{\alpha-1}+\mu h\hat{u}^{q-1}=0.
\end{equation*}
Moreover, we can deduce the formula of $\mu$.
Inserting $v=\hat{u}$ into the equality (\ref{fin}) yields that
\begin{equation*}
\int_V(|\nabla \hat{u}|^{p}-\lambda f\hat{u}^{\alpha}+\mu h\hat{u}^{q})d\mu=\int_V(|\nabla \hat{u}|^{p}-\lambda f\hat{u}^{\alpha})d\mu-\mu q=0.
\end{equation*}
thus $\mu$ can be given by
\begin{equation}\label{const}
\mu=\mu(\lambda,\hat{u})=\frac{1}{q}\int_V(|\nabla \hat{u}|^{p}-\lambda f\hat{u}^{\alpha})d\mu.
\end{equation}

\textbf{Step 4}. $\hat{u}>0$.\\

Recall that $$\Phi=\Phi(t,\mu)=E(u+tv)+\mu G(u+tv),$$
then by direct calculation, we have
\begin{equation}\label{equ-gradient-I-i}
\frac{\partial \Phi}{\partial u_i}=\mu_i(-\Delta_pu_i-\lambda f_iu^{\alpha-1}_i+\mu hu^{q-1}_i).
\end{equation}

Note the graph $G$ is connected, if $\hat{u}>0$ is not satisfied, since $\hat{u}\geq0$ and not identically zero, then there is an edge $i\thicksim j$, such that $\hat{u}_i=0$, but $\hat{u}_j>0$. Then by the definition of the $p$-th Laplacian $\Delta_p$ in section 2 we have
$$-\Delta_p\hat{u}_i=-\frac{1}{\mu_i}\sum\limits_{k\thicksim i}\omega_{ik}|\hat{u}_k-\hat{u}_i|^{p-2}(\hat{u}_k-\hat{u}_i)<0.$$
Therefore by (\ref{equ-gradient-I-i}), we have
$$\frac{\partial \Phi}{\partial u_i}\Big|_{u=\hat{u}}=\mu_i(-\Delta_p\hat{u}_i)<0.$$
Recall we had proved that $\hat{u}$ is the minimum value of $\Phi$, hence there should be
$$\frac{\partial \Phi}{\partial u_i}\Big|_{u=\hat{u}}\geq0,$$
which is a contradiction. Hence $\hat{u}>0$.\par
This completes the proof of Theorem \ref{th1}. $\hfill\Box$

\section{Proofs of Theorem \ref{th2} and Theorem \ref{th3}}
\subsection{Proofs of Theorem \ref{th2}}
\label{sect-preliminary-lemma}
Now we can prove Theorem \ref{th2} by Theorem \ref{th1}.\par
First choose any negative constant $\tilde{\lambda}<0$, and assume that $h,f\in C(V)$ with $h\leq0$, $f>0$, $1\leq\alpha\leq p\leq q$, then consider the following $p$-th Yamabe equation
\begin{equation}\label{L3}
  -\Delta_pu+\mu hu^{q-1}=\tilde{\lambda} fu^{\alpha-1}.
\end{equation}
By Theorem \ref{th1}, we know that this equation on $G$ always has a positive solution $\hat{u}>0$ for some constant $\mu=\mu(\tilde{\lambda})\in \mathbb{R}$.
Equivalently, the solution $\hat{u}>0$ satisfies
\begin{equation}\label{L2}
\int_V(-\Delta_{p}\hat{u}-\tilde{\lambda} f\hat{u}^{\alpha-1}+ \mu h\hat{u}^{q-1})vd\mu=0, \forall v\in C(V).
\end{equation}

Next using this positive solution $\hat{u}>0$, we will construct a positive solution $$u=u(\hat{u},\mu)>0$$ of the the $p$-th Yamabe equation
\begin{equation}\label{E}
\int_V(-\Delta_{p}u-\lambda fu^{\alpha-1}+ hu^{q-1})vd\mu=0, \forall v\in C(V).
\end{equation}
Recall that $\tilde{\lambda}<0$, $f>0, \hat{u}>0$, hence by the formula (\ref{const}) we have
\begin{equation*}
\mu=\frac{1}{q}\int_V(|\nabla \hat{u}|^{p}-\tilde{\lambda} f\hat{u}^{\alpha})d\mu>0.
\end{equation*}
then scaling with a suitable power of $\mu$, we set $$u(s)=\mu^{s}\hat{u},$$
thus $\hat{u}=\mu^{-s}u(s)$, substitute it into the equation (\ref{L2})
\begin{eqnarray*}
0&=&\int_V(-\Delta_{p}\hat{u}-\tilde{\lambda} f\hat{u}^{\alpha-1}+ \mu h\hat{u}^{q-1})vd\mu \\
&=&\int_{V}(|\nabla \hat{u}|^{p-1}\nabla v-\tilde{\lambda} f\hat{u}^{\alpha-1}v+\mu h\hat{u}^{q-1}v)d\mu\\
&=&\int_{V}(\mu^{-s(p-1)}|\nabla u(s)|^{p-1}\nabla v-\mu^{-s(\alpha-1)}\tilde{\lambda} fu(s)^{\alpha-1}v+\mu^{1-s(q-1)}hu(s)^{q-1}v)d\mu\\
&=&\mu^{-s(p-1)}\int_{V}(-\Delta_{p}u(s)-\mu^{-s(\alpha-p)}\tilde{\lambda} fu(s)^{\alpha-1}+\mu^{1-s(q-p)}hu(s)^{q-1})vd\mu.
\end{eqnarray*}
 Let $1-s(q-p)=0$, thus $s=\frac{1}{q-p}$. Then set $$u=u(\hat{u},\mu)=u(s)|_{s=\frac{1}{q-p}}=\mu^{\frac{1}{q-p}}\hat{u},
 \lambda=\lambda(\tilde{\lambda},\mu)=\mu^{-s(\alpha-p)}|_{s=\frac{1}{q-p}}\tilde{\lambda}=\mu^{\frac{p-\alpha}{q-p}}\tilde{\lambda}, .$$
  By the equality above we obtain that $\forall v\in C(V)$
 \begin{eqnarray*}
&&\int_V(-\Delta_{p}u-\lambda fu^{\alpha-1}+ hu^{q-1})vd\mu\\
&=&\mu^{\frac{p-1}{q-p}}\int_V(-\Delta_{p}\hat{u}-\tilde{\lambda} f\hat{u}^{\alpha-1}+ \mu h\hat{u}^{q-1})vd\mu\\
&=&0, 
\end{eqnarray*}
which implies $$-\Delta_{p}u+ hu^{q-1}=\lambda fu^{\alpha-1}.$$
Finally, by $\hat{u}>0, \mu>0$ we conclude that $u=\mu^{\frac{1}{q-p}}\hat{u}>0$.\par
This completes the proof of Theorem \ref{th2}. $\hfill\Box$

\subsection{Proofs of theorem \ref{th3}}
\label{sect-preliminary-lemma}
To prove Theorem \ref{th3}, first we observe that the unique difference between Problem I and Problem II is the definition of $\Delta_{p}$. Therefore the proof in this case is totally similar. We only sketch the idea of the proof here.\par
First, we define the Sobolev form $\|\cdot\|_{W^{1,p}}$ and the integration $F(u)=\int_V |\nabla u|^{p}d\mu$ in energy functional by using definitions (\ref{Gama}) and (\ref{tidu})
\begin{eqnarray}\label{T}
 |\nabla u|_{i}=\sqrt{\Gamma(u,u)_{i}}, \Gamma(u,\phi)_{i}=\frac{1}{2\mu_{i}}\sum\limits_{j\thicksim i}\omega_{ij}(u_{j}-u_{i})(\phi_j-\phi_i),
  \forall i \in V.
\end{eqnarray}
Then we can use this $|\nabla u|$ to define energy functional as follows.
For $\alpha\geq p$, we can define energy functional similarly as in \cite{Ge3}
\begin{equation}
I(u)=\left(\int_V|\nabla u|^pd\mu-\int_Vhu^pd\mu\right)\left(\int_Vfu^{\alpha} d\mu\right)^{-\frac{p}{\alpha}},
\end{equation}
and for $\alpha\leq p$, our energy functional (\ref{EF}) is similarly defined by
\begin{equation*}\label{ef}
  E(u)=\frac{1}{p} \int_V|\nabla u|^pd\mu-\frac{\lambda}{\alpha}\int_Vfu^{\alpha}d\mu,
\end{equation*}
on the space $C(V)$, also restricted to the subset\\
$$M=\{u\in C(V): \frac{1}{q}\int_V hu^{q}d\mu=-1\}.$$

Note that $F(u)=\int_V |\nabla u|^{p}d\mu$ in definitions of two energy functionals above is continuously Fr\'{e}chet differentiable. By the proof of Theorem 5 in \cite{GLY3}, the Fr\'{e}chet derivatives $DF(u)$ is given by
\begin{equation}
<v, DF(u)>=\int_V (-\Delta_{p}u)vd\mu, \forall v\in C(V),
\end{equation}
where $\Delta_{p}$ is exactly the $p$-th Laplacian $\Delta_{p}$ in (\ref{PL}).\par
This equality is key to guarantee that all of proofs in this paper and in \cite{Ge3} can be generalized in this case. This can be checked explicitly for two cases $\alpha\geq p$, $\alpha\leq p$ step by step. $\hfill\Box$\\

\noindent \textbf{Acknowledgements:} The first author would like to thank Professor Yanxun Chang for constant guidance and encouragement. The second author would like to thank Professor Gang Tian and Huijun Fan for constant encouragement and support. Both authors would also like to thank Professor Huabin Ge for many helpful conversations. The first author is supported by National Natural Science Foundation of China under Grant No. 11431003. The second author is supported by National Natural Science Foundation of China under Grant No. 11401578.

Xiaoxiao Zhang: xiaoxiaozhang0408@bjtu.edu.cn\\
Institute of Mathematics, Beijing Jiaotong University, Beijing 100044, P.R. China\\
Aijin Lin: aijinlin@pku.edu.cn\\
College of Science, National University of Defense Technology, Changsha 410073, P. R. China\\
\end{document}